\documentclass[review]{elsarticle}

\usepackage{hyperref}

\journal{Annals of Pure and Applied Logic}

\usepackage{amsmath}
\usepackage{amssymb}
\usepackage{amsthm}
\usepackage{graphicx}
\usepackage{eufrak}
\usepackage{color}
\usepackage{graphpap}
\usepackage[mathscr]{euscript}
\usepackage{maod}

\newcommand{\A}{\mathfrak{A}}

\newcommand{\F}{\mathbb{F}}

\newcommand{\ket}{\rangle}
\newcommand{\bra}{\langle}
\newcommand{\re}{\mathcal{L}(H)}
\newcommand{\real}{\mathbb{R}}

\newcommand{\co}{\mathbb{C}}
\newcommand{\ts}{\dot{-}}

\newcommand{\interior}{\text{\rm int}}
\newcommand{\tst}{\dot{-}}
\newcommand{\dom}{\rm{dom}}

\begin{document}

\begin{frontmatter}

\title{Quantum mechanics in a metric sheaf: a model theoretic approach}

\author{Maicol A. Ochoa\fnref{myfootnote}}
\address{Department of Chemistry, University of Pennsylvania, Philadelphia PA 19401, USA}
\fntext[myfootnote]{maochoad@gmail.com}

\author{Andr\'es Villaveces\fnref{myfootnote2}}
\address{Departamento de Matem\' aticas, Universidad Nacional de Colombia, Bogot\'a 111321, Colombia}
\fntext[myfootnote2]{avillavecesn@unal.edu.co}

%
%

\begin{abstract}
 We study model-theoretical structures for prototypical physical systems. First, a summary of the model theory of sheaves, adapted
  to the metric case, is presented. In particular, we provide conditions for a
  generalization of the  generic model theorem to metric sheaves.  The essentials of the model theory of metric sheaves already appeared in the form of
  Conference Proceedings\cite{Ochoa2016}. We provide a version of those results, for the sake
  of completeness, and then build metric sheaves for physical systems in the 
  second part of the paper. Specifically, metric sheaves for quantum mechanical systems with pure point and continuous spectra are constructed. In the first case, 
  every fiber is a finite projective Hilbert space determined by the family of invariant subspaces of a given operator with pure point spectrum, and we also consider unitary transformations in a finite-dimensional space. In the second case (an operator with continuous spectrum), every fiber is a two sorted structure of subsets of the Schwartz space of rapidly
  decreasing functions that includes imperfect representations of   position and momentum states.
  The imperfection character is parametrically determined by the elements on the base space and refined in the generic model. Position and momentum operators find a simple representation in every fiber as well as their corresponding unitary operators. These results follow after recasting the algebraic properties of the integral transformations frequently invoked in the description of quantum mechanical systems with continuous spectra. Finally, we illustrate how this construction permits the calculation of the quantum mechanical propagator for a free particle. 
\end{abstract}


\end{frontmatter}


\section{Introduction and motivation}


In 1936, John von Neumann and Garrett Birkhoff introduced a
propositional calculus based on the lattice of closed subspaces of a
Hilbert space ordered by inclusion \cite{NEUM}. This lattice is not
Boolean but orthocomplemented, and is therefore different from the
classical propositional calculus. In particular, the classical
distributive law fails. This propositional calculus is meant to
capture essential differences between the classical (Boolean) and
quantum mechanical picture of nature. Many physical and also
philosophical questions have been stated in this context, and
different kinds of quantum logics have been
proposed\cite{bell1986,chiara1989,svozil1996,dalla2002,engesser2011}. For
example, Domenech and Freytes  have presented a contextual logic
\cite{DOME} to investigate to what extent one can refer to physical
objects without contradiction with quantum mechanics. In their work
they introduced a sheaf over the topological space associated with
Boolean sublattices of the ortholattice of closed subspaces of the
Hilbert Space of a physical system. Connections with the
Kochen-Specker theorem were addressed. Later, Abramsky and Mansfield\cite{ABRA}
studied systematically connections between Kochen-Specker phenomena and non-locality, using minimalist sheaf constructions and cohomology to
measure the extent of the non-locality phenomenon.

Independently, Isham~\cite{isham1994} studying the Gell-Mann and Hartle axioms for a generalized  `histories'  approach to quantum theory, arrived to a possible
lattice structure in general history theories, providing also a number of potential models for theories of this type. In later investigations, D\"oring and Isham\cite{doring2008,doring2008B} appealed to sheaf structures, as they concluded that theories in physics are equivalent to representations in a topos of a certain formal language that depends on the system. Classical mechanics, for instance,  arises when the topos is the category of sets. Quantum physics projection operators are associated with the spectral presheaf\cite{doring2008}.

 In yet another alternative approach,
 Zilber\cite{zilber2008,zilber2016} describes the formalism of quantum
 mechanics with geometric semantics using sheaves built out of Zariski geometries\cite{solanki2014,morales2015}. Particular attention is devoted to the structure of the Heisenberg-Weyl algebra generated by position and momentum coordinates in one-dimensional quantum mechanical systems. These results motivated further work by Hirvonen and Hyttinen \cite{Hirvonen2016}, where detailed calculations for the quantum mechanical propagators for a free particle and harmonic oscillator were obtained by using metric ultraproducts of finite-dimensional approximations to the Hilbert space. At the time our work was completed, Bays and Hart were considering the ultralimit of finite dimensional subspaces of the space of tempered distributions to describe $L_2(\mathbb{R})$ in the rigged Hilbert space form (personal communication). All these different approaches have in common that they attempt to understand the true model for quantum mechanics as a limit model of finite substructures or as a very large finite structure whose best mirror is the limit model.
 
In this paper we revisit the question of defining model structures for
prototypical quantum mechanical systems.  The traditional approach to
the description of isolated quantum mechanical systems assumes that
all possible states of a system are well represented by the elements
of a complex Hilbert space. The class of all self-adjoint operators
with domain contained in such Hilbert space is associated with the set
of observables: magnitudes that can be measured through an experiment
in the laboratory. In this sense, the mathematical properties of the
Hilbert space with the structure provided by this class of operators
constitutes a complete physical representation of an isolated quantum
mechanical system. The spectrum of a self-adjoint operator contains
the most likely outcomes of an experimental measurement of the
observable associated with such operators.  In this regard we are
faced with two different scenarios: a quantum mechanical system
defined in separable Hilbert space that admits a basis set in terms of
the eigenvectors of the operator of interest, and one in which the
operator has a continuous spectrum.  The Dirac formulation of quantum
mechanics leads to a unified representation of these two scenarios,
but the nature and structure of the physical Hilbert spaces is
intrinsically different. We will therefore adopt two different
constructions for a quantum mechanical metric sheaf accordingly.

 For the case of an operator with pure point spectrum, quantum
 mechanics can be studied in terms of Projective Hilbert spaces. This
 approach eliminates the redundancy in the association of many vectors
 in a Hilbert space with the same state, i.e., an element $x$ of a
 Hilbert space and any non-vanishing scalar multiple of it represent
 the same physical state. The main disadvantage of this approach is
 that Projective Hilbert spaces inherit neither the vector space nor
 the inner product space structure of the Hilbert spaces. On the
 positive side,  projective Hilbert spaces facilitate a geometrical
 description of quantum mechanics and the study of the local behavior
 of wave functions. Most frequently, the spectrum and the Hilbert
 space are countable and finite substructures from the same are
 considered, allowing for matrix representations of physical observables.
 It is assumed that the properties of these restricted spaces are the
 same as those expected from the infinite dimensional Hilbert
 space. The mathematical literature has plenty of examples with finite
 substructures that do not capture essential properties from the
 infinite structures where they are naturally embedded. We also know
 that the ultraproduct of a family of finite structures can give rise
 to interesting infinite dimensional structures whose theory is well
 described by \L o\'{s}'  theorem.  Thus, in the
 Sec. \ref{sec:discrete} we address the description of infinite
 dimensional projective Hilbert spaces from their finite dimensional
 substructures by defining a sheaf with such substructures as the
 fibers growing on top of a given topological space. The ``spectral
 sheaf'' introduced by  Domenech and Freytes\cite{DOME} resembles the
 sheaf we introduce in Sec. \ref{sec:discrete}, but devoid of the metric structure.

 If the operator has a continuous spectrum, such as in the case of a position and momentum operators,  the space  $\mathcal{L}_2(X, \mu)$  of all continuous square integrable functions defined in the measure space ($X , \mu$) constitutes the Hilbert space for our system, where $X$ is the physical configuration space for the same. For example, for a free particle the configuration space is $X=\real$ with the Borel measure. The formalism introduced by Dirac suggests that one can find representations for position eigenstates in this space, and that the inner product between two such position eigenstates equals the Dirac delta function. This is in stark conflict with the standard definition of an inner product, and we will expand on this subject below. As a resolution, we investigate quantum mechanical systems with continuous spectra by defining a metric sheaf on a subset of the Schwartz space defined on $X$: the space of rapidly decreasing, infinitely differentiable functions $\mathcal{S}(X,\mu)$.  

After this introduction to the physical content of our paper, a few words on
the model theory of sheaves (with metric fibers and continuous
predicates) are in order.  Sheaves can be regarded as supports for cohomology
 constructions and as systems of variable structures themselves. This part of our paper summarizes results
from our earlier work in Ref.\ \cite{Ochoa2016}. For the sake of completeness
we include the following items from that paper.
\begin{itemize}
\item Basic definitions of metric sheaves, together with general discussion.
\item The main lemmas leading to the generic model theorem.
\end{itemize}

We do \emph{not} include the proofs of the generic model theorem or of
some of the lemmas leading to that. The reader may check them in
detail in~\cite{Ochoa2016}.

 The model theory for metric sheaves is partially motivated by Caicedo's
 results~\cite{CAIC}.  We  briefly summarize here some essential aspects of his work.
 Given two topological spaces $X$ and $E$, a
 sheaf over $X$ is defined as the pair $(E,\pi)$ where $\pi:E \to X$ is a
 local homeomorphism. For each $x \in X$, the fiber $E_x=\pi^{-1}(x)$ is
 the universe of a first order structure in a language
 $\mathcal{L}$.  A section $\sigma$ is a
 continuous function defined from an open set $U \subset X$ in $E$
 such that $\sigma \circ \pi$ is the identity map in $U$.  As a
 consequence of these definitions, the image set $\rm{Im}(\sigma)$ is
 an open set in $E$, and sections are in one-to-one correspondence
 with their image sets. Thus the
 satisfaction relation on each fiber can be extended transversally
 along the sheaf, i.e. from fiber to fiber, by defining a forcing
 relation that describes a semantics in the same language
 $\mathcal{L}$ and where the variables can be interpreted in the
 family of sections. Another important property of this construction
 is that whenever a statement is forced in a fiber $E_x$, one can
 always find a neighborhood $U$ of $x$, such that for every $y \in U$
 the same statement is forced in $E_y$. In addition a new
 $\mathcal{L}$-structure (the generic model) is obtained as a quotient
 space and the satisfaction relation is determined by the forcing
 relation defined in the sheaf.  Caicedo connects this generic model
 through a ``generic model theorem'' (or ``forcing theorem'') to
 various other fundamental results in model theory.


 We first take these ideas to the realm
 of continuous logic. In brief, we construct sheaves of metric
 structures as understood and studied in the model theory developed by
 Ben Yaacov, Berenstein, Henson and Usvyatsov \cite{HENS}. In this case,
  logical connectives in metric structures are continuous functions
  from $[0,1]^n$ to $[0,1]$ and the supremum and infimum play the role
  of quantifiers.  Semantics differs from that in classical structures
  by the fact that the satisfaction relation is defined on
  $\mathcal{L}$-conditions rather than on $\mathcal{L}$-formulas,
  where $\mathcal{L}$ is a metric signature. If $\phi(x)$ and
  $\psi(y)$ are $\mathcal{L}$-formulas, expressions of the form
  $\phi(x) \leq \psi(y)$, $\phi(x) < \psi(y)$, $\phi(x) \geq \psi(y)$,
  $\phi(x) > \psi(y)$ are $\mathcal{L}$-conditions. In addition, if
  $\phi$ and $\psi$ are sentences then we say that the condition is
  closed. 

 The
 set $\mathcal{F}=\{0,1,x/2, \tst\}$, where $0$ and $1$ are taken as
 constant functions, $x/2$ is the function taking half of its input
 and $\tst$ is the truncated subtraction, is uniformly dense in the set
 of all connectives\cite{HENS}. We may therefore restrict the set of
 connectives that we use in building formulas to the set
 $\mathcal{F}$. These constitute the set of $\mathcal{F}$-restricted
 formulas.

 We are now in place to define the sheaf of metric structures,
 recall the definitions of pointwise and local forcing on sections and show how to
 define a metric space in families of sections. In section
 \ref{secgenmodel} we show how to construct the metric generic model
 from a metric sheaf. We also show how the semantics of the generic
 model can be understood by the forcing relation and the topological
 properties of the base space of the sheaf. 

 The presentation of the paper is as follows. First, in Sec.\
 \ref{sec:modth}, we provide a summary of results that were presented
 during the Puebla WOLLIC Conference in 2016 and have appeared as
 Conference Proceedings. In particular, we provide the definitions of
 metric sheaves, their semantics (each fiber is a metric continuous
 structure), and state the generic model theorem for that
 context.
 In Sec.\ \ref{sec:discrete} we define a metric sheaf for a countable
 Hilbert space with a self-adjoint operator. In Sec.\
 \ref{projmetsheaf} we define a metric sheaf for a unitary evolution
 operator. Then in Sec.\ \ref{sec:physics} we study a metric sheaf for
 noncommuting observables with continuous spectra. Finally,
in Sec. \ref{sec:conclusion} we summarize and conclude.   

\section{The model theory of metric sheaves}
\label{sec:modth}

This section summarizes the results from our paper\cite{Ochoa2016}: mainly, we
provide the definitions of metric sheaves, the pointwise and local
semantics, the construction of the metric on the sheaf and of generic
models, and finally the generic model theorem. We do not provide the
proofs of these here - the reader is referred to\cite{Ochoa2016} for
details.

Consider a topological space $X$.
 A sheafspace over $X$ is a pair $(E,\pi)$, where $E$ is a
 topological space and $\pi$ is a local homeomorphism from $E$ into
 $X$. A section $\sigma$ is a function from an open set $U$ of $X$ to
 $E$ such that $\pi \circ \sigma =Id_U$. We say that the section is
 global if $U=X$. Sections are determined by their images, as $\pi$ is
 their common continuous inverse function. Besides, images of sections
 form a basis for the topology of $E$. We will refer indistinctly to
 the image set of a section and the function itself. 

  In what follows we assume that a metric language $\mathcal{L}$ is
  given and we omit the prefix $\mathcal{L}$ when talking about
  $\mathcal{L}$-formulas, $\mathcal{L}$-conditions, etc. 

\begin{definition}[Sheaf of metric structures]\label{metricsheafdef}
Let $X$ be a topological space.  A sheaf of metric structures (or, for
short, a ``metric sheaf'') $\A$
over $X$ consists of:
\begin{enumerate}
\item A sheafspace $\: (E, \pi)$ over $X$.
\item For all $x$ in $X$ we associate a metric structure\\
$(\A_x, d) = \left( E_x, \{R_i^{(n_i)}\}_x , \{f_j^{(m_j)}\}_x,
  \{c_k\}_x, \Delta _{R_{i,x}}, \Delta _{f_{i,x}}, d, [0,
  1]\right)$,\\
 where $E_x$ is the fiber $\pi^{-1}(x)$ over $x$, and the following
 conditions hold:
\begin{enumerate}
\item $(E_x, d_x)$ is a complete, bounded metric space of diameter
  $1$.
\item For all $i$, $R_i^\A =\bigcup_{x \in X} R_i^{\A_x}$ is a
  continuous function according to the topology of $\bigcup_{x \in X}
  E_x^{n_j}$. 
\item For all $j$, the function $f_j^\A = \bigcup_x f_j^{\A_x}
  :\bigcup_x E_x^{m_j} \to \bigcup_x E_x$ is a continuous function
  according to the topology of $\bigcup_{x \in X} E_x^{m_j}$.
\item For all $k$, the function $c_k^\A : X \to E$, given by $c_k^\A
  (x)=c_k^{\A_x}$, is a continuous global section.
\item We define the premetric function $d^\A$ by $d^\A=\bigcup _{x \in
    X}d_x: \bigcup _{x \in X} E_x^2 \to [0, 1]$, where $d^\A$ is a
  continuous function according to the topology of $\bigcup_{x \in X} E_x^2$.
\item For all $i$, $\Delta^\A_{ R_i}=\inf_{x \in X}
  (\Delta^{\A_x}_{R_i})$ with the condition that $\inf_{x \in X}
  \Delta^{\A_x}_{R_i}(\varepsilon)>0$ for all $\varepsilon>0$. 
\item For all $j$, $\Delta^\A_{ f_j}=\inf_{x \in X}
  (\Delta^{\A_x}_{f_i})$ with the condition that  $\inf_{x \in X}
  \Delta^{\A_x}_{f_i}(\varepsilon)>0$ for all $\varepsilon>0$.
\item The closed interval $[0,1]$ is a second sort and is provided
  with the usual metric.
\end{enumerate}
\end{enumerate}
\end{definition}

 The space $\bigcup_x E_x^{n}$ has as basic open sets the image of sections
 given by $\bra \sigma_1, \dots, \sigma_n \ket= \sigma_1 \times \dots
 \times \sigma_n \cap \bigcup_x E_x^{n}$. These are the sections of a
 sheaf over $X$ with local homeomorphism $\pi^*$ defined by $\pi^*\bra
 \sigma_1(x), \dots, \sigma_n(x) \ket= x$. We drop the symbol $^*$
 from our notation when talking about this local homeomorphism but it
 must be clear that this local homeomorphism differs from the function
 $\pi$ used in the definition of the topological sheaf.

The induced function $d^\A$ (the ``global distance function'') is not
necessarily a metric nor a
pseudometric. Thus, we cannot expect the sheaf just defined to be a
metric structure, in the sense of continuous logic. Indeed, we want to
build the local semantics on the sheaf so that for a given sentence
$\phi$, if $\phi$ is true at some $x \in X$, then we can find a
neighborhood $U$ of
$x$ such that for every $y$ in $U$, $\phi$ is also true. In order to
accomplish this task, first note that semantics in
continuous logic is not defined on formulas but on conditions. Since
the truth of the condition ``$\phi<\varepsilon$'' for $\varepsilon$
small can be thought as a good approximation to the notion of $\phi$
being true in a first order model, one may choose this as the
condition to be forced in our metric sheaf.  Therefore, for a given
real number $\varepsilon \in (0,1)$, we consider conditions of the
form $\phi < \varepsilon$ and $\phi > \varepsilon$. Our first result
comes from investigating to what extent the truth in a fiber
``spreads'' onto the sheaf.
 
\begin{lemma}[Truth continuity in restricted cases]\label{conti}

\begin{itemize}
\item Let $\varepsilon$ be a real number, $x \in X$, $\phi$ an
  $\mathcal{L}$-formula composed only of the logical metric
  connectives and perhaps the quantifier $\inf$.  If $\A_x \models
  \phi (\sigma(x)) < \varepsilon$, then there exists an open
  neighborhood $U$ of $x$, such that for every $y$ in $U$, $\A_y
  \models \phi(\sigma(y)) < \varepsilon$.
\item Let $\varepsilon$ be a real number, $x \in X$, $\phi$ an
  $\mathcal{L}$-formula composed only of the logical metric
  connectives and perhaps the quantifier $\sup$. If $\A_x \models \phi
  (\sigma(x)) > \varepsilon$, then there exists an open neighborhood
  $U$ of $x$, such that for every $y$ in $U$, $\A_y \models
  \phi(\sigma(y)) > \varepsilon$.
\end{itemize}
\end{lemma}

  In particular, the above lemma is true for $\mathcal{F}$-restricted
  sentences. Lemma \ref{conti} can be proven by induction using
  density. This approach provides the
  setting to define the point-forcing relation on conditions. 

\begin{definition}[Point Forcing]
Given a metric sheaf $\; \A$ over a topological space $X$, we define
the relation $\Vdash_x$ on the set of all conditions of the form $\phi
< \varepsilon$ and $\phi > \varepsilon$ (where $\phi$ is an
$\mathcal{L}$-statement, $\varepsilon$ is an arbitrary real number in
$(0,1)$ and $x\in X$). Furthermore, where in our definition
$\phi=\phi(v_1,\cdots,v_n)$ has free variables,  the
forcing at $x$ will depend on specifying \emph{local sections}
$\sigma_1,\cdots,\sigma_n$ of the
sheaf defined on open sets around the point $x$. Where necessary (for
atomic formulas and the quantifier stage) we will indicate this.

Our definition is by induction on the complexity of
$\mathcal{L}$-statements, and given for every $\varepsilon \in (0,1)$
simultaneously.


Atomic formulas
\begin{itemize}
\item $\A \Vdash _x d(\sigma_1, \sigma_2)<\varepsilon
\iff d_x(\sigma_1(x), \sigma_2(x))<\varepsilon
$
\item
$
\A \Vdash _x R(\sigma_1, \dots, \sigma_n)< \varepsilon 
\iff R^{\A_x}(\sigma_1(x), \dots, \sigma_n(x)) < \varepsilon
$
\item similar to the previous two, but with $>$ instead of $<$
\end{itemize}

Logical connectives
\begin{itemize}
\item
$
\A \Vdash _x \max(\phi, \psi)<\varepsilon
\iff \A \Vdash _x \phi < \varepsilon $ and $ \A \Vdash _x \psi <\varepsilon
$
\item
$
\A \Vdash _x \max(\phi, \psi)>\varepsilon 
\iff \A \Vdash _x \phi > \varepsilon $ or $ \A \Vdash _x \psi > \varepsilon
$
\item
$
\A \Vdash _x \min(\phi, \psi) < \varepsilon
\iff \A \Vdash _x \phi < \varepsilon $ or $ \A \Vdash _x \psi < \varepsilon
$
\item
$
\A \Vdash _x \min(\phi, \psi) > \varepsilon
\iff \A \Vdash _x \phi > \varepsilon $ and $ \A \Vdash _x \psi > \varepsilon
$
\item
$
\A \Vdash _x 1 \tst \phi < \varepsilon
\iff \A \Vdash _x \phi > 1 - \varepsilon 
$
\item
$
\A \Vdash _x 1 \tst \phi > \varepsilon
\iff \A \Vdash _x \phi < 1 - \varepsilon 
$
\item $\A \Vdash _x \phi \tst \psi < \varepsilon \iff \A_x \models \psi=1$ or $\A_x \models \psi = r$ for some $r \in (0,1)$ and one of the following holds:\\
i) $\A \Vdash_x \phi < r $\\
\hspace{2.5cm} ii) $\A \nVdash_x \phi < r$ and $\A \nVdash_x \phi > r$\\
\hspace{2.5cm} iii) $\A \Vdash_x \phi > r $ and $ \A \Vdash_x \phi < r + \delta$ for some $\delta \in (0, \varepsilon)$.
\item $\A \Vdash _x \phi \tst \psi > \varepsilon \iff \A \Vdash_x  \phi > r + \varepsilon$ with $r$ such that $\A_x \models \psi=r$
\end{itemize}

Quantifiers
\begin{itemize}
\item
$
\A \Vdash _x \inf _\sigma \phi(\sigma) < \varepsilon
\iff $There exists a section $\mu$ such that $ \A \Vdash _x \phi (\mu )< \varepsilon$.
\item
$
\A \Vdash _x \inf _\sigma \phi(\sigma) > \varepsilon
\iff $
There exists an open set $U \ni x$ and a real number $\delta_x>0$ such that for every $y \in U$ and every section $\mu$ defined on $y$, $ \A \Vdash _y \phi (\mu )> \varepsilon + \delta_x$
\item
$
\A \Vdash _x \sup _\sigma \phi(\sigma) <\epsilon 
\iff 
$There exists an open set $U \ni x$ and a real number $\delta_x$ 
such that for every $y \in U$ and every section $\mu$ defined on $y$
 $\A \Vdash_y \phi(\mu) < \varepsilon - \delta_x$.
\item
$
\A \Vdash _x \sup _\sigma \phi(\sigma) >\epsilon 
\iff $ There exists a section $\mu$ defined on $x$ such that $
 \A \Vdash_x \phi(\mu) > \varepsilon $
\end{itemize}
\end{definition}

The above definition and the previous lemma lead to the equivalence
between $\A \Vdash _x \inf _\sigma (1 \tst \phi) > 1 \tst \varepsilon $
and $\A \Vdash _x \sup _\sigma \phi < \varepsilon$. More important, we
can state the truth continuity lemma for the forcing relation on
sections as follows.

\begin{lemma}
  Let $\phi(\sigma)$ be an $\mathcal{F}-$restricted formula. Then 
  \begin{enumerate}
  \item $\A \Vdash _x \phi(\sigma) < \varepsilon$ iff there exists $U$
    open neighborhood of $x$ in $X$ such that  $\A \Vdash _y
    \phi(\sigma) < \varepsilon$ for all $y \in U$.
\item $\A \Vdash _x \phi(\sigma) > \varepsilon$ iff there exists $U$
  open neighborhood of $x$ in $X$ such that  $\A \Vdash _y
  \phi(\sigma) > \varepsilon$ for all $y \in U$.
  \end{enumerate}
 \end{lemma}

We can also define the point-forcing relation for non-strict inequalities by
\begin{itemize}
\item $\A \Vdash _x \phi \leq \varepsilon$ iff $\A \nVdash _x \phi > \varepsilon $ and
\item $\A \Vdash _x \phi \geq \varepsilon$ iff $\A \nVdash _x \phi < \varepsilon $,
\end{itemize}
for $\mathcal{F}-$restricted formulas. This definition allows us to show the following proposition.

\begin{proposition}
 Let $0<\varepsilon' < \varepsilon$ be real numbers. Then
 \begin{enumerate}
 \item If $\A \Vdash_x \phi(\sigma) \leq \varepsilon '$ then $\A \Vdash_x \phi(\sigma) < \varepsilon$.
\item If $\A \Vdash_x \phi(\sigma) \geq \varepsilon$ then $\A \Vdash_x \phi(\sigma) > \varepsilon '$.
 \end{enumerate}
\end{proposition}

 The fact that sections may have different domains brings additional
 difficulties to the problem of defining a metric function with the
 triangle inequality holding for an arbitrary triple. However, we do
 not need to consider the whole set of sections of a sheaf but only
 those whose domain is in a filter of open sets (as will be evident in
 the construction of the ``Metric Generic Model'' below). One may
 consider a construction of such a metric by defining the ultraproduct
 and the ultralimit for an ultrafilter of open sets. However, the
 ultralimit may not be unique since $E$ is not always a compact set in
 the topology defined by the set of sections. In fact, it would only
 be compact if each fiber was finite. Besides, it may not be the case
 that the ultraproduct is complete. Thus, we proceed in a different
 way by observing that a pseudometric can be defined for the set of
 sections with domain in a given filter.

\begin{lemma}\label{lemmapseudometric}
Let $\mathbb{F}$ be a filter of open sets. For all sections $\sigma$
and $\mu$ with domain in $\mathbb{F}$, let the family
$\mathbb{F}_{\sigma \mu}=\{U\cap \dom(\sigma)\cap \dom(\mu) | U \in
\F\}$. Then the function
\[
\rho_{\mathbb{F}}(\sigma, \mu)=\inf_{U \in \F _{\sigma \mu}} \;\sup_{x
  \in U} d_x (\sigma(x), \mu(x))
\]
is a pseudometric in the set of sections $\sigma$ such that
$\dom(\sigma ) \in \mathbb{F}$.
\end{lemma}

 \begin{proof} See\cite{Ochoa2016}.
\end{proof}

In the following, whenever we talk about a filter $\mathbb{F}$ in $X$ we will be considering a filter of open sets. For any pair of sections $\sigma$, $\mu$ with domains in a filter, we define $\sigma \sim_\mathbb{F} \mu$  if and only if $\rho_{\mathbb{F}}(\sigma, \mu)=0$. This is an equivalence relation, and the quotient space is therefore a metric space under $d_\mathbb{F}([\sigma], [\mu] )= \rho_\mathbb{F} (\sigma, \mu)$.  The quotient space provided with the metric $d_\F$ is the metric space associated with the filter $\mathbb{F}$. If $\F$ is principal and the topology of the base space X is given by a metric, then the associated metric space of that filter is complete. In fact completeness is a trivial consequence of the fact that sections are continuous and bounded in the case of a $\sigma$-complete filter (if $X$ is a metric space). However, principal filters are not interesting from the semantic point of view and $\sigma$-completeness might not hold for filters or even ultrafilters of open sets. The good news is that we can still guarantee completeness in certain kinds of ultrafilters.

\begin{theorem}\label{regspace}
  Let $\A$ be a sheaf of metric structures defined over a regular
  topological space $X$. Let $\F$ be an ultrafilter of regular open
  sets. Then, the induced metric structure in the quotient space
  $\A[\F]$ is complete under the induced metric.
\end{theorem}



In regards with the above theorem, it is worth noting the following
connection between Cauchy sequences of sections in the pseudometric
$\rho_\F$ and the regularity of the space $X$.

\begin{lemma}\label{cauchylimit}
Let $\F$ be a filter and $\{ \sigma_n\}$ be a Cauchy sequence of
sections according to the pseudometric $\rho_\F$ with all of them
defined in an open set $U$ in $\F$. Then
\begin{enumerate}
\item There exists a limit function $\mu_\infty$ not necessarily
  continuous defined on $U$ such that $\lim_{n \to \infty}
  \rho_\F(\sigma_n, \mu_\infty)=0$.
\item If $X$ is a regular topological space and
  $\interior(\rm{ran}(\mu_\infty))\neq \emptyset$, there exists an open
  set $V \subset U$, such that $\mu_\infty \upharpoonright V$ is
  continuous.
\end{enumerate}
\end{lemma}
\begin{proof} See\cite{Ochoa2016}.
\end{proof}

Before studying the semantics of the quotient space of a generic
filter, we define the relation $\Vdash_U$ of local forcing in an open
set $U$ for a sheaf of metric structures. The definition is intended
to make the following statements about local and point forcing valid
\[  \A \Vdash_U \phi(\sigma) < \varepsilon \iff \forall x \in U \; \A
  \Vdash_x \phi(\sigma)< \delta \mbox{ and } \]
\[
  \A \Vdash_U \phi(\sigma) > \delta \iff \forall x \in U \; \A
  \Vdash_x \phi(\sigma)> \varepsilon,\]
for some $\delta < \varepsilon$. This is possible as a consequence of the truth continuity lemma.

\begin{definition}[Local forcing for Metric Structures]
  Let $\A$ be a Sheaf of metric structures defined in $X$,
  $\varepsilon$ a positive real number, $U$ an open set in $X$, and
  $\sigma_1,\dots, \sigma_n$ sections defined in $U$. If $\phi$ is an
  $\mathcal{F}$- restricted formula the relations $\A \Vdash_U
  \phi(\sigma) < \varepsilon$ and $\A \Vdash_U \phi(\sigma) >
  \varepsilon$ are defined by the following statements

Atomic formulas
\begin{itemize}
\item $\A \Vdash_U d(\sigma_1, \sigma_2) < \varepsilon \iff \sup_{x \in U} d_x(\sigma_1(x), \sigma_2(x)) < \varepsilon$
\item $\A \Vdash_U R(\sigma_1,\dots, \sigma_n) < \varepsilon \iff \sup_{x \in U} R^{\A_x}(\sigma_1(x),\dots, \sigma_n(x)) < \varepsilon$
\item Similar to the previous two, with $>$ instead of $<$ and $\sup$
  replaced by $\inf$
\end{itemize}

Logical connectives
\begin{itemize}
\item $\A \Vdash_U \max(\phi, \psi) < \varepsilon \iff$ $\A \Vdash_V \phi < \varepsilon$ and $\A \Vdash_W \psi < \varepsilon$
\item $\A \Vdash_U \max(\phi, \psi) > \varepsilon \iff$ There exist open sets $V$ and $W$ such that $V \cup W = U$ and $ \A \Vdash_V \phi > \varepsilon$ and $\A \Vdash_W \psi > \varepsilon$
\item $\A \Vdash_U \min(\phi, \psi) < \varepsilon \iff $  There exist open sets $V$ and $W$ such that $V \cup W = U$ and  $\A \Vdash_V \phi < \varepsilon$ and $\A \Vdash_W \psi < \varepsilon$ 
\item $\A \Vdash_U \min(\phi, \psi) < \varepsilon \iff \A \Vdash_U \phi < \varepsilon$ and $\A \Vdash_U \psi < \varepsilon$
\item $\A \Vdash_U 1 \tst \psi < \varepsilon \iff \A \Vdash_U \psi> 1 \tst \varepsilon$
\item $\A \Vdash_U 1 \tst \psi > \varepsilon \iff \A \Vdash_U \psi< 1 \tst \varepsilon$
\item $\A \Vdash_U \phi \tst \psi < \varepsilon \iff $ One of the following holds\\
i) There exists $r \in (0,1)$ such that $\A \Vdash_U \phi < r $ and $\A \Vdash_U \psi > r$\\
\hspace{2.5cm} ii) For all $r \in (0,1)$, $\A \Vdash_U \phi < r$ if and only if $\A \Vdash_U \psi < r$\\
\hspace{2.5cm} iii) $\A \Vdash_U \phi < \varepsilon$ \\
\hspace{2.5cm} iv) There exists $r, q \in (0,1)$ such that\\
\hspace{4.0cm} $\A \Vdash_U \phi > r$ and $\A \Vdash_U \psi < r$ \\
\hspace{4.0cm} $\A \Vdash_U \phi < q+ \varepsilon$ and $\A \Vdash_U \psi > q$ \\
\hspace{4.0cm} and for all $\delta < \varepsilon$ and $\A \Vdash_U \phi > \delta$
\item $\A \Vdash_U \phi \tst \psi > \varepsilon \iff$ There exists $q>0$  such that $ \A \Vdash_U \psi < q$ and $\A \Vdash \phi > q + \varepsilon$
\end{itemize}

Quantifiers
\begin{itemize}
\item $\A \Vdash_U \inf_\sigma \phi(\sigma) < \varepsilon \iff $ there
  exist an open covering $\{U_i\}$ of $U$ and a family of section
  $\mu_i$ each one defined in $U_i$ such that $\A \Vdash_{U_i}
  \phi(\mu_i) < \varepsilon$  for all $i$
\item $\A \Vdash_U \inf_\sigma \phi(\sigma) > \epsilon \iff $   there
  exist $\varepsilon'$ such that $0< \varepsilon< \varepsilon'$  and
  an open covering $\{U_i\}$ of $U$ such that for every section
  $\mu_i$ defined in $U_i$ $\A \Vdash_{U_i} \phi(\mu_i) >
  \varepsilon'$
\item $\A \Vdash_U \sup_\sigma \phi(\sigma) < \varepsilon \iff $ there
  exist $\varepsilon'$  such that $0< \varepsilon'< \varepsilon$  and
  an open covering $\{U_i\}$ of $U$ such that for every section
  $\mu_i$ defined in $U_i$ $\A \Vdash_{U_i} \phi(\mu_i) <
  \varepsilon'$
\item $\A \Vdash_U \sup_\sigma \phi(\sigma) > \varepsilon \iff $ there
  exist an open covering $\{U_i\}$ of $U$ and a family of section
  $\mu_i$ each one defined in $U_i$ such that $\A \Vdash_{U_i}
  \phi(\mu_i) > \varepsilon$  for all $i$
\end{itemize}
\end{definition}

Observe that the definition of local forcing leads to the equivalences
\[ \A \Vdash _U \inf _\sigma (1 \tst \phi(\sigma)) > 1 \tst \varepsilon
 \iff \A \Vdash _U \sup _\sigma \phi(\sigma) < \varepsilon,\]
\[\A \Vdash _U \inf _\sigma (\phi(\sigma)) < \varepsilon  \iff \A
  \Vdash _U \sup _\sigma (1 \tst \phi(\sigma)) > 1 \tst \varepsilon.\]

 The fact that we can obtain a similar statement to the
 Maximum Principle of~\cite{CAIC} is even more important.

\begin{theorem}[The Maximum Principle for Metric structures]
If $\A \Vdash_U \inf_{\sigma} \phi(\sigma) < \varepsilon$ then there
exists a section $\mu$ defined in an open set $W$ dense in $U$ such
that $\A \Vdash_U \phi(\mu) < \varepsilon'$, for some $\varepsilon' <
\varepsilon$.
\end{theorem}


\subsection[Metric Generic Model Theorem]{The Metric Generic Model and its theory}\label{secgenmodel}

 In certain cases, the quotient space of the metric sheaf can be the
 universe of a metric structure in the same language as each of the
 fibers. We examine in this section one such case: sheaves of metric
 structures over \emph{regular} topological spaces, and a generic for
 filter of \emph{regular} open sets. We do not claim that this is the
 optimal situation - however, we provide a proof of a version of a
 Generic Model Theorem for these Metric Generic models.

\begin{definition}[Metric Generic Model]
Let $\A=(X,p,E)$ be a sheaf of metric structures defined on a regular
topological space $X$ and $\F$ an ultrafilter of regular open sets in
the topology of $X$. We define the Metric Generic Model $\A[\F]$ by
\begin{equation}
\A [\F]=\{ [\sigma]/ _{\sim_\F} | \dom(\sigma) \in \F\},
\end{equation}
provided with the metric $ d_\F $ defined above (see Lemma
\ref{lemmapseudometric} and subsequent discussion), and with
\begin{itemize}
\item 
\begin{equation}
f^{\A [\F]}([\sigma_1]/ _{\sim_\F}, \dots, [\sigma_n]/
_{\sim_\F})=[f^{\A}(\sigma_1, \dots, \sigma_n)]/_{\sim_\F}
\end{equation}
with modulus of uniform continuity $\Delta_{f}^{\A[\F]}=\inf_{x \in X}
\Delta_{f}^{\A_x}$.
\item
\begin{equation}
 R^{\A [\F]}([\sigma_1]/ _{\sim_\F}, \dots,
 [\sigma_n]/_{\sim_\F})=\inf_{U \in \F_{\sigma_1 \dots \sigma_n}}\:
 \sup_{x \in U} R_x(\sigma_1(x), \dots, \sigma_n(x))
\end{equation}
with modulus of uniform continuity $\Delta_{R}^{\A[\F]}=\inf_{x \in X}
\Delta_{R}^{\A_x}$.
\item 
\begin{equation}
c^{\A [\F]}=[c]/_{\sim_\F}
\end{equation}
\end{itemize}
\end{definition}

 Observe that the properties of $d_\F$ and the fact that $R^{\A}$ is
 continuous ensure that the Metric Generic Model is well defined as a
 metric structure. An important observation, also developed in full
 detail in\cite{Ochoa2016} is that in our setting, $R^{\A[\F]}$ is
 indeed uniformly continuous.

 It is worth mentioning that part of the ``generality'' of
 the so called generic model is lost. This is indeed true and it is a
 consequence of the additional conditions that we have imposed on the
 topology of the base space (regularity) and on the ultrafilter to
 obtain a Cauchy complete metric space.

 We can now present the Generic Model Theorem (GMT) for metric
 structures. This provides a nice way to describe the theory of the
 metric generic model by means of the forcing relation and topological
 properties of the sheaf of metric structures. 

\begin{theorem}[Metric Generic Model Theorem]\label{genmodtheometric}
Let $\F$ be an ultrafilter of regular open sets on a regular
topological space $X$ and $\A$ a sheaf of metric structures on
$X$. Then
\begin{itemize}
\item
\begin{equation}
  \A[\F] \models \phi([\sigma]/_{\sim \F})< \varepsilon \iff \exists U
  \in \F \mbox{ such that } \A \Vdash_U \phi(\sigma)< \varepsilon
\end{equation}
\item
  \begin{equation}
     \A[\F] \models \phi([\sigma]/_{\sim \F})> \varepsilon \iff
     \exists U \in \F\mbox{ such that }\A \Vdash_U \phi(\sigma)>
     \varepsilon
  \end{equation}
\end{itemize}
\end{theorem}

We now stress that the Metric Generic Model Theorem (GMT) has distinct
but strong connections with the Classical Theorem (see \cite{CAIC,
  FORE}).  In the case of the Metric GMT, we can observe similarities
in the forcing definitions if we consider the parallelism between the
minimum function and the disjunction, the maximum function  and the
conjunction, the infimum and the existential quantifier. On the other
hand, differences are evident if we compare the supremum with the
universal quantifier. The reason for this is that in this case the
sentence $1 \tst (1 \tst \phi)$, which is our analog for the double
negation in continuous logic, is equivalent to the sentence
$\phi$. Note that the point and local forcing definitions are
consistent with this fact - i.e.,
\[ \A \Vdash_U 1 \tst (1 \tst \phi)< \varepsilon \iff \A \Vdash_U \phi <
\varepsilon ,\]
\[ \A \Vdash_U 1 \tst (1 \tst \phi)> \varepsilon \iff \A \Vdash_U \phi >
\varepsilon .\]
As another consequence, the metric version of the GMT does not require
an analog definition to the G\" odel translation. 

 We close this section by introducing a simple example that
 illustrates some of the elements just described. We study the metric
 sheaf for the continuous cyclic flow in a torus.

 Let $X= S^1$, $E= S^1 \times S^1$ and $p = \pi _1$, be the projection
 function onto the first component. Then, we have $E_q= S^1$. Given a
 set of local coordinates $x_i$ in $S_i$ and a smooth vector field $V$
 on $E$, such that
 \[  V= V_1\frac{\partial }{\partial x_1}+V_2\frac{\partial}{\partial
     x_2}\]
\[ V_1(p) \neq 0 \;\; \forall p \in S^1,\]
 we can take as the set of sections the family of integrable curves of
 $V$. The open sets of the sheaf can be described as local streams
 through $E$. Complex multiplication in every fiber is continuously
 extended to a function between integral curves. Every section can be
 extended to a global section.

  Let us study the metric generic model of this sheaf. Note that $X$
  is a topological regular space and that it admits an ultrafilter
  $\F$ of regular open sets. First, observe that $\A[\F]$ is a proper
  subset of the set of local integrable curves. In fact, every element
  in $\A [\F]$ can be described as the equivalence class of a global
  section in $E$:
For any element $[\sigma]\in  \A[\F]$, $U = dom (\sigma) \in \F$, and
there exists a global integral curve $\mu$ in $E$ such that
$\rho_\F(\sigma, \mu)=0$. This result leads to the conclusion that
every ultrafilter filter of open sets in $S^1$ generates the same
universe for $\A[\F]$. 
  Observe that every fiber can be made into a metric structure with a
  metric given by the length of the shortest path joining two
  points. This, of course, is a Cauchy complete and bounded metric
  space. Dividing the distance function by $\pi$, we may redefine this
  to make $d(x,y)\leq 1$, for $x$ and $y$ in $S^1$. Therefore, this
  manifold is also a metric sheaf. In addition, observe that complex
  multiplication in $S^1$ extends to the sheaf as a uniformly
  continuous function in the set of sections.
  For any element $[\sigma] \in \A [\F]$, let $U= dom (\sigma )\in \F$
  and $\mu$ be the global integral curve that extends $\sigma$. Thus,
  for arbitrary $\varepsilon >0$
\begin{equation}
  \A \Vdash _U d^\A(\sigma, \mu) < \varepsilon
\end{equation}
and as a consequence
\[ \A[\F] \models d^{\A[\F]}([\sigma],
  [\mu]) = 0. \]
 In addition, the metric generic model satisfies the condition that
 {\sl multiplication} between sections is left continuous. Let
 $\eta$ and $\mu$ be sections whose domain is an element of the
 ultrafilter. For any $\varepsilon < 1/2$, if 
\[
\A \Vdash _{\dom(\eta)\cap \dom(\mu)} d(\eta, \mu) < \varepsilon
\]
then for any other section $\sigma$ defined in an element of $\F$, it
is true that in $V=\dom(\eta)\cap \dom(\mu)\cap \dom(\sigma)$ 
\begin{equation}
\A \Vdash _V  d(\eta \sigma, \mu \sigma) < \varepsilon
\end{equation}
and also
\begin{equation}
\A \Vdash _V  1 \tst \max (d(\eta, \mu), 1 \tst d(\eta \sigma, \mu
\sigma)) < \varepsilon.
 \end{equation}
By the metric GMT, we can conclude that
\[
\A[\F]\models 1 \tst \max (d^{\A[\F]}([\eta], [\mu]), 1 \tst d([\eta]
[\sigma], [\mu] [\sigma])) < \varepsilon
\]
and since $\sigma, \eta$ and $\mu$ were chosen arbitrarily.
\[
\A[\F]\models \sup_\sigma \sup_\eta \sup_\mu \bigl[ 1 \tst \max
(d^{\A[\F]}([\eta], [\mu]), 1 \tst d([\eta] [\sigma], [\mu]
[\sigma]))\bigr] < \varepsilon .
\]
Right continuity, left invariance and right invariance of this metric
can be expressed in the same way.

\section{Metric Sheaf for an infinite-dimensional projective Hilbert space}\label{sec:discrete}

 In this section we construct a sheaf model for a projective Hilbert space when the Hilbert Space has a countable orthonormal basis, with a self-adjoint operator whose domain is the full space. Infinite dimensional projective spaces are realized as the direct limit of finite dimensional ones. Here, we present such a construction in the context of model theory for metric structures\cite{Ochoa2016,HENS}.

 We define the sheaf for the lattice of finite subspaces of a Hilbert space as follows.

We fix a Hilbert space $H$, a self-adjoint bounded operator $A$ defined on $H$ and a maximal set $\{x_k | k \in K \}$ of pairwise orthogonal eigenvectors of $A$.
\begin{enumerate}
\item ({\sl Base space $X$}) Let $I \subset K$ be a finite index set and define
$l_I=\{x_i | i  \in I\}$. Also, let $X=\re=\{l_I | I \subset K \text{ and } |I|< \aleph_0 \}$ be the lattice of finite subsets of the fixed set of eigenvectors. We partially order $X$ by
\[
l_I \prec l_J \mbox{ iff } I\subset J.
\]
The topology on $X$ is generated by the 
basic open sets
\[[l)=\{l' \in \re | l \prec l' \}.\]
\item ({\sl Fibers of the Sheaf}) Every fiber is a two-sorted topological structure 
\begin{equation*}
E_I=\left( PV_I, J_I;E_{A_I}, A_I, (E_{K_\alpha})_\alpha, (K_\alpha)_\alpha, P \right) 
\end{equation*} 
where  $\alpha$ ranges over an index set $\Lambda$.
\begin{enumerate} 
\item $PV_I$ is the finite-dimensional complex projective space associated to $l_I$, constructed as follows. First, for $l_I \in \re$ define on the vector space $ V_I$ spanned by $l_I$ in the complex field the equivalence relation $\sim$ by
\begin{equation*}
y \sim x \iff \exists c \in \co \setminus \{0\} \text{ such that } y = cx,
\end{equation*}
 for $x$ and $y$ different to $0$. Let $[y]$ be the equivalence class with representative element $y$ and   
\begin{equation*}
PV_I= \{ [y] | y \in V_I\setminus \{0\}  \}.
\end{equation*}
Thus, $PV_I$ is the complex projective space of $V_I$. The second sort is the closed interval $J_I$ corresponding to its numerical range, i.e., $J_I= \{\bra x , A_I x \ket : ||x||=1 \}$ with the standard metric of the real set.

\item We provide $PV_I$ with the Fubini-Study metric. This is a K\"ahler metric given by
\[
d([x], [y])= \arccos \sqrt{\frac{\bra x, y \ket \bra y, x \ket }{\bra x, x \ket \bra y, y \ket}},
\]
where $\bra x, y \ket$ is the inner product in the Hilbert space. This corresponds to the length of the geodesic in the finite dimensional sphere in $V_I$ connecting $y/ ||y||$ and $x/ ||x||$.
\item The complex projective sort has symbols $A_I$, $E_{A_I}$ and $P$ to be interpreted as follows.  First, notice that $\tilde A_I = A \upharpoonright V_I$ is a well defined linear operator in $V_I$. Then in $PV_I$, we interpret the expected value function $E_{A_I}$ for $A_I$ as 
  \begin{equation*}
   E_{A_I}([x]) = \frac{\bra \tilde A_I x, x \ket}{\bra x, x \ket}.
  \end{equation*}
 This is a function from the projective sort to the complex numbers. $A_I$ is the function symbol from the projective sort to itself associated with the operator $\tilde A_I$ and defined by 
\begin{equation*}
  A_I[x]= [ \tilde A_I x ].
\end{equation*}
 $P$ is a binary function symbol from the projective sort to the complex numbers, interpreted as follows
\begin{equation*}
P([x],[y])=\frac{\bra x , y \ket \bra y, x \ket}{\bra x, x\ket \bra y, y \ket}.
\end{equation*}
 This can be regarded as the square of the projection of an equivalence class $[x]$ into the equivalence class $[y]$. We have also included into the language additional symbols $K_\alpha$ and  $E_{K_\alpha}$ to be interpreted in a similar way than $A_I$ and $E_{A_I}$ for a linear operator $K$ with trivial kernel defined in $V_I$. 
\end{enumerate}

\item ({\sl Sheaf topology })  If a metric sheaf is such that every fiber is a multisorted structure, the topology for every sort should be given by sections in such a way that every function from one sort to another is continuous. Thus, we must extend the definition for a sheaf in Ref. \citenum{Ochoa2016} to include the following statement. 

Let $S_{x 1}, \dots, S_{x n}$ be sorts of a model $E_x$, then the function $f^{\A}_j= \cup_x f^x_j: \cup_x S_{x n_1}\times \dots \times S_{x n_k} \to \cup_x S_{x m}$ must be continuous.

Thus, consider again our sheaf and let $[l_I )$ be a basic open set in $\re$. We define for each $[x]$ in $PV_I$ the function
\begin{align}\label{secciones}
\sigma_{x} : & [l_I) \to \bigsqcup_{I \subset J} PV_J \\
\sigma_{x} &(l_J)=[x]
\end{align}
and for each $c = (c_1, \dots, c_{n_I})$ in $\mathbb{C}^{n_I}$ 
\begin{align}\label{seccionest}
\mu_{c} : & [l_I) \to \bigsqcup_{I \subset J} J_I\\
\mu_{c} &(l_K)= \frac{\sum |c_i|^2 \lambda_i^I }{\sum|c_i|^2}
\end{align}
where $l_K \in [l_I)$ and  $\{\lambda_i^I \}$ is the set of real eigenvalues $A_I$. These are the sections of our sheaf.
\end{enumerate}
From the set of sections just defined it is clear that there are no global sections in the projective sort. Before showing that the sheaf is well defined, we need to stress that it is not possible to define an inner product in $PV_I$ as a function of the inner product in $V_I$ only. However, in the absence of an inner product we choose $P$ as a geometric descriptor for the projective sort. We say that $[x], [y]$ are orthogonal if $P([x],[y]) < \varepsilon$ for every $\varepsilon \in (0,1)$. In this case they are also orthonormal. From the definition for the metric,  we see that $P([x],[y])$ can be regarded as an angle between two elements in  $PV_I$.

 To show that the topological sheaf is well defined observe that $(A^{\A })^{-1}(\sigma_x)$ is a section for every $\sigma_x$ by analyzing the inverse pointwise. To see that $E_A^\A$ is continuous, for every $c \in \co$ define 
\begin{equation*}
S_c(I)= \left\{[x] | c=\frac{\bra A_I^*x, x \ket}{\bra x, x \ket} \right\},
\end{equation*}
observe that $E_A^{\A\; -1}(\mu_c)= \cup_I S_c(I)$ and that this set is a union of sections. Finally define 
\begin{equation*}
  S'_c(I) =\left\{([x],[y])| c= \frac{\bra x , y \ket \bra y, x \ket}{\bra x, x\ket \bra y, y \ket} \right\},
\end{equation*}
and observe that $P^{\A -1}(\mu_c)=\cup_I S'_c(I)$ is open in $\bigcup_{[l_I)}PV_I^2$. Modulus of uniform continuity can be equally defined.


 We claim that the generic model is an appropriate projective model for the description of many quantum mechanical systems. The structure of the generic model is not described by an application of the Metric Generic Model since the base space is not a regular topological space. Nevertheless, we can take the completion of the generic model and extend the interpretation of functions and relations accordingly.

\begin{theorem}
Let $\A[\F]$ be the metric generic model associated with the sheaf defined above. The following are statements satisfied by this model.
\begin{enumerate}
\item \label{infinitetop} There is an infinite number of orthonormal elements.
\item The function $A^{\A[\F]}$ has an infinite number of eigenvalues.
\item $A$ is not a bounded operator.
\item  $A^{\A [\F]}$ is continuous.   
\end{enumerate}
\end{theorem}

\begin{remark}
We may define the ``dimension'' of the projective sort in every fiber as the maximum number of mutually orthogonal elements. With this definition, statement \ref{infinitetop} in the above Theorem says that our generic model is infinite dimensional. 
\end{remark}
 \begin{remark}
The fact that $A$ is bounded operator and at the same time continuous does not contradict the fact that these two properties of an operator are equivalent in a Hilbert space. (Neither the fibers nor the generic model are vector spaces.)
\end{remark}

\begin{proof}
We only prove the first statement since the other cases are similar.  

Clearly, it is enough to show that for every natural number $n$ we may find an open set $[l_I )$ in $\F$ such that the corresponding fiber $E_I$ has at least $n$ many orthogonal elements. If $I$ has $n+1$ elements, then we have that $\A$ forces the condition
\begin{align*}
  \max (P([\sigma_i],[\sigma_j])| i,j \in n ) <& \varepsilon
\end{align*}
for every $\varepsilon$.


\end{proof}

\section[Projective Metric Sheaf]{A Metric Sheaf  for a projective  Hilbert space with a unitary evolution operator}\label{projmetsheaf}
In this section we consider the case of a quantum mechanical system described by an operator $A_R$ that parametrically depends on the variable $R$ ($R$ will be interpreted as an element of a suitable topological space $X$). A particular realization will be the action of a unitary evolution operator on a Hilbert space, in which case the operator is time-dependent. Another instance will describe a system with space-dependent interactions. In the former case $R$ is time variable, while in the last this corresponds to a spatial coordinate. We intend to build our structures starting from projective Hilbert spaces just as in section \ref{sec:discrete} and we assume $R$ is an element of a regular topological space $X$. If a solution to the eigenvalue problem of $A_R$ exists for every possible value of $R$, we want to extend such a solution continuously to a neighborhood in $X$. The sheaf is therefore defined as follows
\begin{enumerate}
\item ({\sl Base space} $X$)
 Let $R$ be an element of $X$. We choose $X$ to be a space with a basis of regular open sets.
\item ({\sl Fibers of the Sheaf})
Every fiber is a two sorted topological structure
\begin{equation*}
E_R=( PV_R, I_R; A_R, ||A_R||,  E_{A_R}, P )
\end{equation*}
where (see section~\ref{sec:discrete})
\begin{enumerate}
\item $PV_R$ is the complex projective space of a Hilbert space $V_R$. The space $V_R$ is the domain of a self-adjoint and bounded  finite operator $A_R$ that depends parametrically on $R$. Every operator has associated with itself the real number $||A_R||$ that represents its norm and a closed interval $I_R$ corresponding to its numerical range, i.e., $I_R= \{\bra x , A_R x \ket : ||x||=1 \}$. The last is our second sort with the standard metric of the real set. We can also introduce an operator in $PV_R$ with the same name by
  \begin{align*}
    A_R&:PV_R \to PV_R\\
    &A_R[x]=[A_R x]
\end{align*}
Since $PV_R$ is not a vector space, we may not expect $A_R$ to be linear.
\item We provide $PV_R$ with the Fubini-Study metric as defined in Sec. \ref{sec:discrete}
\item $E_{A_R}$ is a function symbol from the projective sort to the interval $I_R$ interpreted as follows
  \begin{equation*}
    E_{A_R}([x])=\frac{\bra A_R x , x \ket}{\bra x, x \ket}
  \end{equation*}
\end{enumerate}
\item ({\sl Modulus of uniform continuity for} $E_{A_R}$)
The modulus of uniform continuity $\Delta_{E_{A_R}}$ for $E_{A_R}$ is obtained under the assumption that $A_R$ is bounded. In this case,  $A$ is also uniformly continuous with respect to the Euclidean metric. Observe that,
\begin{align*}
  ||A_R||&= \sup_{||x||=1 \: ||y||=1} |\bra y, A_R x \ket |\\
\intertext{Let $x= y +h$, then}
||A_R y -A_R x ||&=||A_R h||\leq ||A_R|| ||h||
\intertext{Now we show that the function $E_{A_R}$ is uniformly continuous with respect to the Euclidean metric of the Hilbert space. It is enough to consider $x$ and $y$ of unit length:}
||\bra A_R y,y\ket -\bra A_R x, x\ket ||&=||\bra A_R (x + h),x+h\ket -\bra A_R x, x\ket ||\\
&=||\bra A_R h,  h\ket+\bra A_R h, x\ket+\bra A_R x, h\ket ||\\
&\leq ||h||^2 \: ||\bra A_R h/||h||,h/||h||\ket||+\\
&\hspace{8mm}||h||\:||\bra A_R h/||h||, x\ket||+||h||\: ||\bra A_R x, h/||h||\ket ||\\
&\leq ||h||(||h|| + 2)||A_R||=\varepsilon_R\\
\intertext{given that $||y-x||=||h||$ we can take as a modulus of uniform continuity }
\delta(\varepsilon) &= \sqrt{1+\frac{\varepsilon}{||A_R||}}-1.
\end{align*}
It can be shown that this also implies that $E_A[x]$ is uniformly continuous as a function from $PV_R$ to $I_R$. We use the fact that $d_{FS}([x],[y])$ equals in magnitude the angle of the geodesic that connects two points $x/ ||x||$ and $y/ ||y||$ in the unit sphere $S^n$. Thus we can define the modulus of uniform continuity respect to $d_{FS}$ according to 
\begin{equation*}
\Delta_{E_{A_R}}(\varepsilon_R)= \arccos \left(  1 - \frac{\delta^2(\varepsilon_R)}{2}\right) 
\end{equation*}
 Similar arguments can be used to show that there is a modulus of uniform continuity for $A^\A$.

\item ({\sl The topology of the sheafspace}) For the projective sort, we define sections in such a way that they respect the basis of the initial Hilbert space $V_R$.
Let $\{x_i^R\}$ be a basis of orthonormal eigenvectors of $A_R$ ordered according to
\begin{equation*}
  x_i^R \leq x_j^R \iff E_{A_R}[x_i^R] \leq E_{A_R}[x_j^R].
\end{equation*}
In case of $ E_{A_R}[x_i^R] = E_{A_R}[x_j^R]$, choose any consistent ordering. Given a regular open set $U$ in $X$, and $c=[c_i ] \in \mathbb{PC}^n$ we define a section $\sigma_c$ as
\begin{align*}
  \sigma_c:&U \to E_{PV_R}\\
\sigma_c(R)&=[\sum c_i x_i^R]
\end{align*}
For the numerical range sort, we define sections as follows 
\begin{align*}
   \mu_c:&U \to E_{I_R}\\
\mu_c(R)&=\frac{\sum |c_i|^2 \lambda_i^R }{\sum|c_i|^2}
\end{align*}
where $\{\lambda_i^R \}$ is the set of real eigenvalues $A_R x_i^R = \lambda_i^R x_i^R$.
\end{enumerate}

It remains to show that $E^\A$ and $P^\A$ are well defined, i.e., that they are continuous. This follows after noticing that $E_A^{\A -1}(\mu_c)= \sigma_c$. 
Next, we investigate a few elements from the theory of the sheaf just defined. Consider first the formula 
\[
\phi_{norm}^{\A_R}=\inf_{[x]}\bigl| E_{A_R}[x]-||A_R||\bigr|
\]
which is a statement about $||A_R||$. In our sheaf $\A \Vdash_R \phi_{norm} < \varepsilon $ for all $\varepsilon >0$. If there is $\sigma_{norm}$, a global section such that $\sigma_{norm}(R)=||A_R||$ for all $R \in X$, then $[\sigma_{norm}]/_{\sim F}$ is a constant in $\A[\F]$ and $\phi_{norm}^{\A[\F]} < \varepsilon$ is true in $\A[\F]$ as well. In that case we can also state that $||A_R||$ is unique. However, note that in this case the norm is not a real number but a section.

In the same fashion, the condition

\begin{align*}
 d_{FS}(A_R[\sigma_{1,0,\dots,0}]/_{\sim\F},[\sigma_{1,0,\dots,0}]/_{\sim\F})&< \varepsilon\\
\end{align*}
is forced in dom$(\sigma_{1,0, \dots, 0})$ and this tells us that $[\sigma_{1,0,\dots,0}]$ is an eigenvector of $A_R$. Also, the condition
and 
\begin{align*}
\bigl| E_A^{\A [\F]}[\sigma_{1,0,\dots,0}]/_{\sim\F} - [\mu_{1,0,\dots,0}]/_{\sim\F} \bigr|< \varepsilon
\end{align*}
identifies the section, if any, in the numerical range sort that should be interpreted as the corresponding eigenvalue.

 We may not have enough tools in our language to let the model know about the dimension of the projective spaces $PV_R$ through local isomorphism to open subsets of $\mathbb{R}^n$. However, the generic model knows about its dimension by means of the projective function $P$. The analogy comes from a classical model of an inner product space. As an example, consider the first order sentence in the language of an inner product space of ``dimension 2'':
\begin{multline*}
\phi_{\text{dim 2}}^{classic}= \exists x_1 \exists x_2 \Bigl(x_1 \neq x_2 \land \bra x_1,x_2 \ket =0 \land \\
\forall x_3 (x_1 \neq x_3 \land x_1 \neq x_3 \land \bra x_1, x_3 \ket \neq 0 \land \bra x_2, x_3 \ket\neq 0 )\Bigr)   
\end{multline*}
we write the analog metric sentence for our model, as follows
\begin{multline*}
  \phi_{\text{dim 2}}^{metric}=
 \inf_{\sigma_1} \inf_{\sigma_2} \max \Bigl(1 \ts d_{FS}(\sigma_1,\sigma_2) ,\: P(\sigma_1, \sigma_2),\\
 \sup_{\sigma_3}\bigl( \max (1\ts d_{FS}(\sigma_1,\sigma_3),\: 1\ts d_{FS}(\sigma_1,\sigma_3),\: 1 \ts P(\sigma_1, \sigma_3),\: 1\ts P(\sigma_2, \sigma_3)\:)\bigr)\Bigr) 
\end{multline*}
and the condition $ \phi_{\text{dim 2}}^{metric}< \varepsilon$ should be forced in a fiber if dim$(PV_R)=2$. As a trivial consequence, the same condition should be satisfied by the generic model if the dimension of $V_R$ is the same for all $R \in X$. A more interesting problem is to find a sheaf whose generic model is infinite dimensional and all whose fibers are finite dimensional. The following lemmas state the properties of the sheaf associated with an infinite dimensional metric generic model. 
  
\begin{lemma}
  The generic model for the above sheaf is infinite dimensional if and only if for every $k\in \omega$, the family $U_k=\{R \in X \: | \: \A \Vdash_R \phi_{dim > k}< 2^{-k}  \}$ is a subset of $\F$ such that $\cap U_k= \emptyset$. 
\end{lemma}
\begin{proof}
  ($\Rightarrow$) Suppose that $\A[\F]\models \phi_{dim>n} < \varepsilon $ for any $\varepsilon > 0$. Then it is true that there is a family $\{U_k \mid  k \in \omega\}$ such that $\A \Vdash _{U_k} \phi_{dim>k} < 2^{-k}$ and we might assume that $U_{k+1} \supset U_{k}$. Thus, $\cap_k U_k = \emptyset$ otherwise, there would exist an infinite fiber and this would contradict the lemma of truth continuity.

($\Leftarrow$) Given $\varepsilon > 0$, there exists $k$ such that $2^{-k}< \varepsilon$ and then $\A \Vdash_{U_k} < \varepsilon $. Thus $\A[\F] \models  \phi_{dim>n} < \varepsilon$ for any $\varepsilon$.  
\end{proof}

\section{A Metric Sheaf for noncommuting observables with continuous spectra.}\label{sec:physics}

 The following discussion will concentrate on the quantum mechanics for position and momentum operators, but can be extended to other observables with continuous spectrum.
 
 The axiomatic framework of quantum mechanics dictates that every {\sl observable} must be described by a self-adjoint operator acting in an appropriate Hilbert Space. Thus, we expect to find for position and momentum operators $\hat x $ and $\hat p$, with domain in the Hilbert space of the system, representing such observables\cite{dirac1981}. Using Dirac's notation, the existence of elements in the Hilbert space, denoted by  $| x \rangle$ and $| p \rangle$, such that the eigenvalue equations $\hat x | x \rangle = x | x \rangle$ and $\hat p | p \rangle = p | p \rangle$ hold, with $x, p \in \mathbb{R}$, is claimed. For many systems $x$ and $p$ can take any value in a measurable subset of $\mathbb{R}$ and therefore we call $\hat x$ and $\hat p$ operators with continuous spectrum.

 The structure of the physical Hilbert space provided with these operators differs from the standard axiomatic definition adopted in Analysis. In particular, the inner product for the physical Hilbert space is not only complex valued, but can take values on the space of distributions. Using again Dirac's notation, one defines the inner product of two position eigenstates $|x_o \rangle$, $|x_1 \rangle$ by
 \begin{equation}
   \label{eq:dotx}
   \langle x_o | x_1 \rangle = \delta (x_o - x_1),
 \end{equation}
where $\delta(x)$ is the Dirac's delta function. In the same way 
\begin{align}
     \langle p_o | p_1 \rangle =& \delta (p_o - p_1),
\end{align}
which implies that neither position nor momentum eigenstates can be normalized. In addition, the inner product between position and momentum eigenstates is given by 
\begin{align}
     \langle p_o | x_o \rangle =& \frac{1}{\sqrt{2 \pi \hbar}}e^{-i x_o p_o / \hbar}.
\end{align}
  The physical Hilbert space has two basis sets: $\{ |x\rangle | x \in \mathbb{R} \}$ and $\{ |p\rangle | p \in \mathbb{R} \}$ which are related to each other by the relations  
\begin{align}
  |p \rangle  = \frac{1}{\sqrt{2 \pi \hbar}} \int dx e^{-i x_o p_o / \hbar} |x \rangle \label{eq:FTp},\\
  |x \rangle  = \frac{1}{\sqrt{2 \pi \hbar}} \int dp e^{+i x_o p_o / \hbar} |p \rangle \label{eq:FTx},  
\end{align}
Letting $\hat I$ be the identity operator for the physical Hilbert space and $[A,B]=AB-BA$, we find that $[\hat x,\hat p]= i \hbar \hat I$. This result has as a consequence that the observables $\hat p$ and $\hat x$ cannot be simultaneously measured in the lab with absolute accuracy (Heisenberg's uncertainty principle). In the basis of position eigenstates, the representation for position and momentum operators is given by  
  \begin{align}
    \hat x  \to& M_x , \label{eq:xrep}\\
    \hat p  \to& -i \hbar \frac{\partial}{\partial x} \label{eq:prep},
  \end{align}
where $M_x$ is the multiplication operator by the constant $x$. Thus $\hat p$ is a differential operator in $\mathcal{L}^2(\real,\mu)$ and therefore is only defined in a subset of the whole Hilbert Space. We can also find a representations for these operators in the basis given by the momentum eigenstates in which case $\hat x$ is a differential operator and $\hat p$ a multiplication operator.

  The work from Laurent Schwartz on distributions\cite{schwartz1959} helped to clear up the notions that Dirac had introduced in the conceptual framework of quantum mechanics and that we have sketched above. In the subsequent discussion, we will find that a model-theoretical picture for the quantum mechanics of position and momentum operators is possible in a subset of the Schwartz space which we define next 
\begin{definition}
The Schwartz space on $\real ^n$ is the function space given by
\begin{equation*}
 S\left(\real ^{n}\right)=\left\{f\in C^{\infty }(\real ^{n}):\|f\|_{\alpha ,\beta }<\infty \quad \forall \alpha ,\beta \in \mathbb {N}^{n}\right\}
\end{equation*}
where $\alpha$, $\beta$ are multi-indices, and $C^\infty(\real ^n)$ is the set of smooth complex valued functions from $\real^n$, and
\begin{equation*}
\|f\|_{\alpha ,\beta }=\sup _{x\in \real ^{n}}\left|x^{\alpha }D^{\beta }f(x)\right|.
\end{equation*}
\end{definition}

 The sheaf that we are going to introduce, finds its motivation in the following definition for the Dirac's distribution in $\mathcal{L}^2(\mathbb{R})$ :
  \begin{equation}
    \label{eq:diracd}
    \lim_{\tau \to 0} \frac{1}{\tau \sqrt \pi} e^{-x^2/\tau^2} = \delta (x) ,
  \end{equation}
with the limit taken in the sense of distributions. This suggests that an {\sl imperfect} representation $\phi_{\tau}(x, x_o)$ for the physical vector state $| x_o \rangle$ in $\mathcal{L}^2(\mathbb{R})$ is
\begin{equation}
  \label{eq:phix}
  \phi_{\tau}(x, x_o) = \frac{1}{\tau \sqrt{ 2 \pi \hbar}} e^{-(x - x_o)^2/2\hbar \tau^2}.
\end{equation}
The family of elements $\{ \phi_{\tau}(x, x_o) \}$ is a subset of the Schwartz space and, with the inner product in $\mathcal{L}^2(\real)$, we find that
\begin{equation}
  \label{eq:dotprod}
  \langle \phi_{\tau}(x, x_o), \phi_{\tau}(x, x_1) \rangle = \int dx \phi_{\tau}(x, x_o) \phi_{\tau}(x, x_1) = \phi_{2 \tau}(x_1, x_o),
\end{equation}
 and moreover the Fourier Transform acts on this family as
 \begin{align}
   FT(\phi_{\tau}(x, x_o)) =&  \frac{1}{\sqrt{2 \pi \hbar}} \int dx e^{-i (x - x_o) (p - p_o) / \hbar}\phi_{\tau}(x, x_o) = \phi_{1/\tau}(p, p_o).
 \end{align}
We will find useful to replace the parameter $\tau$ in Eq.\ \eqref{eq:phix} by the pair $(\tau',t)$ in order to introduce time dynamics, letting $\tau$ range in $\real$ while $t$ will be an element in $\co$. Thus we define
\begin{equation}
  \label{eq:phixt}
  \phi_{(\tau, t)}(x, x_o) = \frac{1}{\sqrt{ 2 \pi \hbar (\tau^2 +t)}} e^{-(x - x_o)^2/2\hbar(\tau^2+t)}.
\end{equation}

Consequently, the physical {\sl inner product} in Eq.\ \eqref{eq:dotx}  and the Fourier transform in Eq.\ \eqref{eq:FTp} can be defined as a set of linear transformations in the family of Gaussian $\phi_{(\tau,t)} (x, x_o)$, if we accept some imperfection in the representation given by the parameter $\tau$. In order to find a representation for the operators $\hat x$ and $\hat p$, we will need to introduce a larger subspace of the Schwartz space that is closed under differentiation. The basis set for such vector space $\mathcal{U}_{\tau}(\mathbb{R})$ is given by  
\begin{align}
    \mathcal{U}_{\tau}(\mathbb{R}) =& \{q(x-x_o)\phi_{(\tau,t)}(x,x_o) : x, x_o \in \mathbb{R} , t \in \co  \},
\end{align}
with $q(x-x_o)$ representing a polynomial of arbitrary order. From Eqs.\ \eqref{eq:xrep} and \eqref{eq:prep} we find that $\hat x$ and $\hat p$ can be represented by the linear transformations
\begin{align} 
  \hat x :& \langle \mathcal{U}_{\tau)}(\mathbb{R}) \rangle \to \langle \mathcal{U}_{\tau}(\mathbb{R}) \rangle\\
& \hat x \left[q(x-x_o)\phi_{(\tau,t)}(x,x_o)\right] = (x-x_o)q(x-x_o)\phi_{(\tau,t)}(x,x_o)\\
  \hat p :& \langle \mathcal{U}_{\tau}(\mathbb{R}) \rangle \to \langle \mathcal{U}_{\tau}(\mathbb{R}) \rangle\\
& \hat p \left[q(x-x_o)\phi_{(\tau,t)}(x,x_o)\right] = -i \hbar q'(x-x_o)\phi_{(\tau,t)}(x,x_o) \notag\\
&\hspace{4cm}+ \frac{i}{\sqrt{\tau^2+t}}(x- x_o)q(x-x_o)\phi_{(\tau,t)}(x,x_o)), 
\end{align}
where we have denoted by $ \langle \mathcal{U}_{\tau}(\mathbb{R}) \rangle$ the space spanned by this basis.
Notice by direct computation on an arbitrary element that these definitions imply that $[\hat x,\hat p]=i \hbar \hat I$, with $\hat I$ the identity operator. 

We now introduce representations for the operators $e^{i t \hat x}$ and $e^{i t \hat p}$. The definition to be used next, follows from the properties of the operator $e^{i t M_x}$ and the unitary equivalence, via Fourier transformations, between $-i \hbar \frac{\partial}{\partial x}$ and $M_p$  \cite{Conway1990}. First let us introduce a new vector space with basis
\begin{equation}
  \mathcal{V}_{\tau} (\mathbb{R}) = \{ q(p - p_o) \phi_{1/{(\tau,t)}}(p, p_o) : p, p_o \in \mathbb{R} , t \in \co\}
\end{equation}
with 
\begin{equation}
  \phi_{1/{(\tau,t)}}(p, p_o) = \sqrt{\frac{(\tau^2 +t)}{ 2 \pi \hbar}} e^{-(\tau^2+t)(p - p_o)^2/2\hbar},
\end{equation}

such that the Fourier Transform is represented as a function between the spaces spanned by $\mathcal{U}_{\tau}(\real) $ and $\mathcal{V}_{\tau}(\real)$ by
\begin{align}
  FT:& \mathcal{U}_{\tau}(\mathbb{R}) \to \mathcal{V}_{\tau}(\mathbb{R}) \label{eq:FTUtoV}\\
  &FT(q(x-x_o)\phi_{(\tau,t)}(x,x_o)) = \frac{1}{\sqrt{\tau^2+t}} q(p-p_o)\phi_{1/{(\tau,t)}}(p,p_o).
\end{align}
In order to describe the unitary operators that depend on position and momentum operators we define
\begin{align}
  e^{itf(\hat x)} \phi_{(\tau,t)}(x, x_o) =& e^{itf((x-x_o))} \phi_{(\tau,t)}(x, x_o) \label{eq:unitaryx}\\
  e^{itf(\hat p)} \phi_{(\tau,t)}(x, x_o) =& FT^{-1} e^{itf((p-p_o))} FT \phi_{(\tau,t)}(x, x_o) \label{eq:unitaryp}
\end{align}
where $f(x)$ is a continuous function on $x$. All these properties are natural from the vector spaces that we have defined and can be proven after integration in $\mathcal{L}^2(\real)$. Importantly, the result of these integral transformations is recast in simple transformations between these families of Gaussians. By adapting these results, we can define a metric sheaf for a free particle in a one dimensional space (such that classical phase-space is $\mathbb{R}^2$ with physical states represented by points in phase-space of the form $(x, p)$).

\begin{definition}
 The triple $\mathfrak{A} =(E,X,\pi)$ where
\begin{itemize}
\item $X = \mathbb{R}^+$ is the base space with the product topology.
\item For $\tau \in X$ we let $E_{\tau}$ be a two sorted metric model where
\begin{itemize}
\item $\mathcal{U}_{\tau}$ and $\mathcal{V}_{\tau}$ span the universe for each sort. 
\item Every sort has is a metric space with the metric induced by the norm in $\mathcal{L}^2(\real)$.
\item Every sort is a model in the language of a vector space, with symbols for the {\sl inner product transformation} $\langle,\rangle_{\mathcal{V}}$ and $\langle,\rangle_{\mathcal{U}}$, to be interpreted such that
\begin{align}
\langle q(x_o-x)\phi_{(\tau, t_1)}(x_o-x),& r(x_1-x) \phi_{(\tau,t_1)}(x_1-x) \rangle_{\mathcal{U}} \notag\\
 =& q(x_o-x_1)r(x_o-x_1)\phi_{(\tau, t_1+t_2)}(x_o-x_1)\\ 
\langle q(p_o-p)\phi_{1/(\tau, t_1)}(p_o-p),& r(p_1-p)\phi_{1/(\tau,t_1)}(p_1-p) \rangle_{\mathcal{V}} \notag\\
 =& q(p_o-p_1)r(p_o-p_1) \phi_{1/(\tau, t_1+t_2)}(p_o-p_1) 
\end{align}

\item function symbols for $FT$ and $FT^{-1}$ to be interpreted as in Eq.\ \eqref{eq:FTUtoV}.
\end{itemize}
\item The sheaf is constructed as the disjoint union of fibers: $E = \sqcup_{\tau \in X} E_{\tau}$ 
\item Sections are defined such that if $\tau \in U \subset X$, 
\[
\sigma_{q,x_o,p_o, t}(\tau) = \left(q(x-x_o)\phi_{(\tau,t)}(x,x_o)\, ,\, q(p-p_o)\phi_{1/(\tau,t)}(p,p_o)\right).
\]
\item $\pi$, the local homeomorphism, is given by $\pi(\psi) = \tau $ if $\psi \in E_\tau$.
\end{itemize}
\end{definition}
 A few remarks are in order. First, the inner products $\langle , \rangle_{\mathcal{U}}$ and $\langle , \rangle_{\mathcal{V}}$ are not the objects frequently defined as the inner product in a Hilbert space. Instead, they are our representation for the physical inner product as defined by Dirac in each sort. Second, notice that in contrast with expression in  Eq.\ \eqref{eq:dotprod}, we enforce with our definition that image element of the mappings $\langle , \rangle_{\mathcal{U}}$ and $\langle , \rangle_{\mathcal{V}}$ preserves the imperfection parameter $\tau$. By doing this the inner products are well defined on every section. Third, since $X$ is a regular topological space it admits a nonprincipal filter of regular open sets.  We are particularly interested in two kinds of generic metric models. In the first kind we look at generic models that capture the limit of vanishing $\tau$, for which we take the nonprincipal ultrafilter induced by the family of open regular sets $\{(0, 1/n) : n \in \mathbb{N} \}$. From the structure of the sheaf defined above, limit elements in the generic model coming from the $\mathcal{U}$ sort with $t =0$ must approach to the Dirac's delta in position. Complementarily, the generic metric model that we obtain when we take the nonprincipal ultrafilter induced by the family of open regular sets $\{(n, \infty) : n \in \mathbb{N} \}$ must contain limit elements that represent Dirac's distributions in momentum space.   

 Next, we show how this metric sheaf permits the computation of the quantum mechanical amplitude for a free particle, overcoming drawbacks found in previous works\cite{Hirvonen2016}. The energy eigenstates of a physical systems in quantum mechanics are characterized by the Hamiltonian operator $\hat H$ and, in the case of a free particle this corresponds to 
\begin{equation}
  \label{eq:Hfreeparticle}
  \hat H= \frac{\hat p ^2}{2 m}. 
\end{equation}
In terms of this operator, we may define the quantum mechanical propagator $K(x_1, x_o, t)$ for a free particle that ``travels'' from $x_o$ to $x_1$ in configuration space by
\begin{align}
K(x_1, x_o, t) =& \langle x_1 ,U(t) x_o \rangle \label{eq:propfree}
\intertext{with}
U(t) =& e^{-it \hat H/ \hbar} = e^{-it  \hat p ^2/ 2 m \hbar}   
\end{align}
We can calculate at each fiber $E_{\tau}$ the propagator $K(x_1, x_o, t)$ in Eq.\ \eqref{eq:propfree} as follows
\begin{align}
  |x_o \rangle  =&  \phi_{{(\tau,0)}}(x,x_o) \\
  U(t) | x_o \rangle =& e^{-it  \hat p ^2/ 2 m \hbar} \phi_{(\tau,0)}(x,x_o) \\
  =& FT^{-1} e^{-it  M_p ^2/ 2 m \hbar} FT \phi_{(\tau,0)}(x,x_o) \\
  =& FT^{-1} e^{-it  M_p ^2/ 2 m \hbar} \left( \frac{1}{\tau}\phi_{1/(\tau,0)}(p,p_o)\right) \\ 
  =& FT^{-1} \left( e^{-it (p-p_o) ^2/ 2 m \hbar} \frac{1}{\tau}\phi_{1/(\tau,0)}(p,p_o)\right) \\
  =& FT^{-1} \frac{1}{\sqrt{2 \pi}} e^{-(\tau^2+it/m) (p-p_o) ^2/ 2 \hbar}\\
  =& \frac{1}{\sqrt{2 \pi( \tau^2+it/m)}} e^{-(x-x_o) ^2/ 2 \hbar (\tau^2+it/m)} \\
  =& \phi_{(\tau,it/m)}(x,x_o)\\
\langle x_1,  U(t) x_o \rangle =& \langle \phi_{\tau}(x,x_1), \phi_{(\tau,it/m)}(x,x_o) \rangle_{\mathcal{U}}\\
=& \phi_{(\tau,it/m)}(x_1,x_o)\\
=& \frac{1}{\sqrt{2 \pi( \tau^2+it/m)}} e^{-(x_1-x_o) ^2/ 2 \hbar (\tau^2+it/m)} \label{eq:ApproxProp}
\end{align}

The result in Eq.\ \eqref{eq:ApproxProp} is the {\sl imperfect} propagator at the fiber $E_\tau$. If we were to take the limit $\tau \to 0$ in this expression  we will recover the \underline{exact} form for the quantum mechanical amplitude, and this is precisely what our choice of the ultrafilter in the base space does: We take the nonprincipal ultrafilter induced by the family of open regular sets $\{(0, 1/n) : n \in \mathbb{N} \}$.  Thus in the Generic model $\mathfrak{A}[\mathbb{F}]$ we recover the exact propagator as a limit element.

 Our calculation for a free particle overcomes some of the
 difficulties found in Ref.\citenum{Hirvonen2016} in two aspects: 1)
 the normalization of the propagator is correct (c.f. page 21 in that
 reference.) 2) The arbitrary scaling factors introduced in the time
 evolution operators (c.f. bottom part of page 19), necessary in their
 derivation to obtain the right form for the propagator are not
 required here. Furthermore, even though their construction is fine
 mathematically speaking, it appears to miss the following
 essential aspect of physics: the ``eigenstates'' in position an
 momentum are not normalizable and the physical inner product can take
 values in the space of distributions. In contrast, one advantage of
 their formulation is that (unlike our formulation in terms of metric sheaves),
 they obtain a simple
 derivation for the propagator for the harmonic oscillator, in the
 cited reference (achieved at the expense of very complex
 scaling relations; c.f. bottom part page 23, Ref.\
 \citenum{Hirvonen2016}).


\section{Conclusion}\label{sec:conclusion}
We have revisited the basic definitions and properties of the model
theory of metric sheaf structures, with a particular interest on the
connections between local forcing on the sheaf and the satisfaction
relation in the metric generic model. The last connection is fully
described by the Generic Model Theorem.  These results have been
successfully applied to the study of physical systems in two cases:
quantum mechanical systems represented by a Hermitian operator with
pure point spectrum and quantum systems with two noncommuting
operators with a continuous spectrum. For an operator with pure point
spectrum we introduced a sheaf that describes, via forcing on
finite-dimensional projective Hilbert spaces, the structure of the
operator in a metric model for an infinite-dimensional projective
Hilbert space (The Generic Model). For a system with two noncommuting
operators, we specialized to the case of position and momentum. After
describing the properties of the physical Hilbert space in such
setting, we constructed a metric model for a free particle from a
metric sheaf where two subsets of the Schwartz space are the universes
in each two-sorted fiber. We found that the generic metric model has
as a limit element the \textbf{quantum mechanical propagator}
for this system, and that its properties can be studied via forcing on
the metric sheaf. Further work in the structure of the metric sheaf
will look for accurate computations of propagators of other more
complex systems, including the harmonic oscillator.

\bibliography{msheafQM}

\begin{thebibliography}{10}
\expandafter\ifx\csname url\endcsname\relax
  \def\url#1{\texttt{#1}}\fi
\expandafter\ifx\csname urlprefix\endcsname\relax\def\urlprefix{URL }\fi
\expandafter\ifx\csname href\endcsname\relax
  \def\href#1#2{#2} \def\path#1{#1}\fi

\bibitem{Ochoa2016}
M.~A. Ochoa, A.~Villaveces, Sheaves of Metric Structures, Springer Berlin
  Heidelberg, Berlin, Heidelberg, 2016, pp. 297--315.

\bibitem{NEUM}
G.~Birkhoff, J.~Von~Neumann, The logic of quantum mechanics, Annals of
  mathematics (1936) 823--843.

\bibitem{bell1986}
J.~L. Bell, A new approach to quantum logic, British Journal for the Philosophy
  of Science (1986) 83--99.

\bibitem{chiara1989}
M.~L.~D. Chiara, R.~Giuntini, Paraconsistent quantum logics, Foundations of
  Physics 19~(7) (1989) 891--904.

\bibitem{svozil1996}
K.~Svozil, J.~Tkadlec, Greechie diagrams, nonexistence of measures in quantum
  logics, and {K}ochen--{S}pecker-type constructions, Journal of Mathematical
  Physics 37~(11) (1996) 5380--5401.

\bibitem{dalla2002}
M.~L. Dalla~Chiara, R.~Giuntini, Quantum logics, in: Handbook of philosophical
  logic, Springer, 2002, pp. 129--228.

\bibitem{engesser2011}
K.~Engesser, D.~M. Gabbay, D.~Lehmann, Handbook of Quantum Logic and Quantum
  Structures: Quantum Structures, Elsevier, 2011.

\bibitem{DOME}
G.~Domenech, H.~Freytes, Contextual logic for quantum systems, Journal of
  Mathematical Physics 46~(1) (2005) 012102.

\bibitem{ABRA}
R.~S. S.~Abramsky, S.~Mansfield, {The {C}ohomology of {N}on-Locality and
  {C}ontextuality}, 8th. Int. Workshop on QPL. EPTCS 95~(74).

\bibitem{isham1994}
C.~J. Isham, Quantum logic and the histories approach to quantum theory,
  Journal of Mathematical Physics 35~(5) (1994) 2157--2185.

\bibitem{doring2008}
A.~D{\"o}ring, C.~J. Isham, A topos foundation for theories of physics: I.
  formal languages for physics, Journal of Mathematical Physics 49~(5) (2008)
  053515.

\bibitem{doring2008B}
A.~D{\"o}ring, C.~J. Isham, A topos foundation for theories of physics: Ii.
  daseinisation and the liberation of quantum theory, Journal of Mathematical
  Physics 49~(5) (2008) 053516.

\bibitem{zilber2008}
B.~Zilber, A class of quantum zariski geometries, LONDON MATHEMATICAL SOCIETY
  LECTURE NOTE SERIES 349 (2008) 293.

\bibitem{zilber2016}
B.~Zilber, The semantics of the canonical commutation relation, arXiv preprint
  arXiv:1604.07745.

\bibitem{solanki2014}
V.~Solanki, D.~Sustretov, B.~Zilber, The quantum harmonic oscillator as a
  {Z}ariski geometry, Annals of Pure and Applied Logic 165~(6) (2014)
  1149--1168.

\bibitem{morales2015}
J.~A.~C. Morales, B.~Zilber, The geometric semantics of algebraic quantum
  mechanics, Phil. Trans. R. Soc. A 373~(2047) (2015) 20140245.

\bibitem{Hirvonen2016}
{\AA}.~Hirvonen, T.~Hyttinen, On {E}igenvectors and the {F}eynman {P}ropagator,
  arXiv preprint arXiv:1407.2134.

\bibitem{CAIC}
X.~Caicedo, L{\'o}gica de los haces de estructuras, Rev. Acad. Colomb. Cienc
  19~(74) (1995) 569--586.

\bibitem{HENS}
I.~B. Yaacov, A.~Berenstein, C.~W. Henson, A.~Usvyatsov, Model theory for
  metric structures, in: the Lecture Notes series of the London Mathematical
  Society, 2007.

\bibitem{FORE}
A.~Forero, {Una demostraci\'on alternativa del teorema de ultral\'imites},
  {Revista Colombiana de Matem\'aticas} 43~(2).

\bibitem{dirac1981}
P.~A.~M. Dirac, The principles of quantum mechanics, no.~27, Oxford university
  press, 1981.

\bibitem{schwartz1959}
L.~Schwartz, I.~de~math{\'e}matique (Strasbourg), Th{\'e}orie des
  distributions, Vol.~2, Hermann Paris, 1959.

\bibitem{Conway1990}
J.~B. Conway, A course in functional analysis, Vol.~96, Springer Science \&
  Business Media, 2013.

\end{thebibliography}

\end{document}